\documentclass[12pt]{article}

\usepackage{amstext}
\usepackage{amsthm}
\usepackage{amsmath}
\usepackage{latexsym}
\usepackage{amsfonts}

\newtheorem{theorem}{Theorem}
\newtheorem{lemma}[theorem]{Lemma}
\newtheorem{claim}[theorem]{Claim}
\newtheorem{proposition}[theorem]{Proposition}
\newtheorem{corollary}[theorem]{Corollary}
\theoremstyle{remark}
\theoremstyle{example}

\def\vol{\operatorname{vol}}

\def\rank{\operatorname{rank}}

\newcommand\CC{{\mathbb C}}

\newcommand\DD{{\mathbb D}}
\newcommand\TT{{\mathbb T}}

\title{Composition operators on the polydisc}
\author{\L ukasz Kosi\'nski
	\thanks{Partially supported by the NCN grant SONATA BIS no. 2017/26/E/ST1/00723}
}
\date{\today}

\begin{document}

\maketitle

\section{Abstract}

In this paper we study the boundedness of composition operators on the weighted Bergman spaces and the Hardy space over the polydisc $\DD^n$. Studying the volume of sublevel sets we show for which $n$ the necessary conditions obtained by Bayart are sufficient. For arbitrary polydisc we prove the rank sufficiency theorem which, in particular, provides us with a simple criterion describing boundedness of composition operators on the spaces over the bidisc. Such a consistent characterization is obtained for the classical Bergman space over the tridisc.

\section{Introduction and statement of the result}

Let $X$ be a Banach space of holomorphic functions on a domain $D\subset \DD^n$ and $\phi$ be a holomorphic self-mapping of $D$. The composition operator, denoted by $C_\phi$, is 
$$f \mapsto f\circ \phi,\quad f\in X.$$
The basic question that matters is for which $\phi$ the symbol $C_\phi$ defines a bounded operator on $X$.

Within the paper we are interested in weighted Bergman and Hardy function spaces over the polydisc. Let us recall their definitions. Let $dA$ denote the normalized measure over the unit disc $\mathbb D$. For $\beta>-1$ we put $dA_\beta (z) = (\beta+1) (1-|z|^2)^\beta dA(z)$. On the polydisc $\DD^n$ one defines $$dV_\beta(z) = dA_\beta(z_1) \ldots dA_\beta (z_n), \quad z=(z_1,\ldots, z_n)\in \DD^n.$$ By $d\sigma$ we denote the normalized surface measure on the torus $\TT^n$.

Then the Bergman space $A_\beta^2(\DD^n)$, where $\beta>-1$, consists of $f\in \mathcal O(\DD^n)$ such that 
$$||f||^2_{A^2_\beta (\DD^n)} = \int_{\DD^n} |f(z)|^2 d V_\beta(z) < \infty.$$
Thus $A^2(\DD^n) := A_0^2(\DD^n)$ is the classical Bergman space.

The Hardy space $H^2(\DD^n)$ is composed of functions $g\in \mathcal O(\DD^n)$ such that 
$$||g||^2_{H^2(\DD^n)} = \sup_{0<r <1} \int_{\TT^n} |g(r \xi)|^2 d \sigma(\xi) <\infty.$$

The problem of continuity of composition operators on Bergman and Hardy spaces over the polydisc has been studied in \cite{KSZ} and then revisited and partially corrected in \cite{Bay}. The second paper constitutes the main motivation for our research. Roughly speaking, the essential result of Bayart is that the problem of continuity of composition operators on the mentioned function spaces involves \emph{delicate} properties of the derivative of the symbol. He showed that if $\phi:\DD^n \to \DD^n$ extends holomorphically to a neighbourhood of $\overline \DD^n$, then the symbol $\phi$ needs to satisfy certain conditions to make $C_\phi$ continuous. Since they rely only on the first derivative, we shall call them the {\it first order} or sometimes \emph{linear conditions}. The second name is justified by the fact that Bayart showed that the first order conditions are sufficient for linear maps. Precisely, carrying out complicated and clever computations he found an algorithm that classified all linear maps $\phi:\DD^n\to \DD^n$ such that $C_\phi$ was bounded. The second main result of \cite{Bay} is a general condition sufficient for continuity of $C_\phi$ over the classical Bergman space $A(\DD^n)$ if the symbol is in $\mathcal O(\DD^n, \DD^n)\cap \mathcal C^2(\overline \DD^n)$. This allowed Bayart to formulate a necessary and sufficient condition that guaranteed continuity of the composition operator on the classical Bergman space $A^2(\DD^2)$, when the symbol was a holomorphic self-map of the bidisc extending holomorphically to a neighbourhood of its closure.

\bigskip 

Let us outline some main results of our paper.
We start with achieving upper and lower bounds for volumes of sublevel sets $\{z\in \DD^n:\ |f(z)-\eta|<\delta\}$, where $f\in \mathcal O(\DD^n,\DD)$ is sufficiently smooth on $\overline \DD^n$ and $\eta\in \TT$. Saying briefly we show that these volumes can be estimated by $\delta^{n+1}$ from above and $\delta^{(3n+1)/2}$ from below. We also show, indicating particular examples, that these bounds are sharp. Equipped with this tool we show that the first order conditions are in general not sufficient. In particular, we prove the following:
\begin{itemize}
	\item first order conditions are sufficient for $A^2(\DD^n)$ if and only if $n\leq 3$;
	\item first order conditions are sufficient for $A^2_\beta(\DD^n)$, $-1<\beta<0,$ (or for $H^2(\DD^n)$) if and only if $n\leq 2$.
\end{itemize}
To obtain sufficiency in the above statements we shall show
\begin{itemize}
	\item the rank-sufficiency theorem for $A_\beta^2(\DD^n)$ and for $H^2(\DD^n)$.
\end{itemize}
The above-mentioned rank theorem implies that the answer in two dimensional settings has a simple formulation:
\begin{itemize}
	\item $\phi\in \mathcal O(\DD^2, \DD^2)\cap \mathcal C(\overline \DD^2)$ defines a bounded composition operator on $A_\beta^2(\DD^2)$ with $\beta>-1$ or on $H^2(\DD^n)$ if and only if the derivative $d_\zeta \phi$ is invertible for all $\zeta\in \TT^2$ such that $\phi(\zeta)\in \TT^2$.
\end{itemize} 
It turns out that a similar simple formulation can be achieved for the classical Bergman space over the tridisc.

Most of the results obtained in the sequel work also for more general spaces $A^p_\beta$ and $H^p$, due to Lemma~5 in \cite{KSZ}.
\bigskip

Amid tools that will be involved in the paper let us mention the characterization of boundedness in terms of volumes of Carleson boxes (see e.g. Lemma~5 in \cite{KSZ}) and its generalization that allowed Bayart to pass from $A_\beta$ to $H^2$ as $\beta \to -1$. Let us recall these ideas. For $\delta = (\delta_1,\ldots, \delta_n)\in (0,1)^n$ and $\xi\in \TT^n$ the Carleson box is the set $S(\xi, \delta)=\{ z\in \DD^n:\ |z_j - \xi_j|<\delta_j,\ j=1,\ldots, n\}$. A holomorphic mapping $\phi:\DD^n\to \DD^n$ defines a bounded operator $C_\phi$ on $A_\beta^2(\DD^n)$ if and only if
\begin{equation}\label{eq:box} V_\beta (\phi^{-1} (S(\xi, \delta))) \leq C_\beta V_\beta (S(\xi, \delta)),\quad \text{for all $\xi\in \TT^n$ and $\delta\in (0,1)^n$},
\end{equation} where $C_\beta$ is a positive constant. It was additionally shown in \cite{Bay}, Proposition~9.3, that if $C_\beta$ can be chosen independently of $-1<\beta<0$, then \eqref{eq:box} entails boundedness of $C_\phi$ on the Hardy space $H^2(\DD^n)$.
 
In the paper we shall need the classical Julia-Carath\'eodory theorem and its multidimensional formulation in the form stated in \cite{Bay} (Lemma~2.5 and Corollary~2.7) as well as other classical tools like the Schwarz lemma etc. We shall also use the following very well known observation that is a consequence of the Montel theorem: if $\varphi\in \mathcal O(\DD^2)\cap \mathcal C(\overline \DD^2)$, then $\varphi(z,\cdot)\in \mathcal O(\DD)$ for any $z\in \TT$.  To avoid introducing too many constants, $f(x)\preceq g(x)$ will always mean that $f(x)\leq C g(x)$ for some positive $C$ that is independent of $x$. The meaning of the symbol $\simeq$ is analogous. $\nabla$ denotes the complex gradient. Moreover, if $I=\{i_1,\ldots, i_k\}\subset \{1,\ldots, n\}$, where $1\leq i_1<\ldots<i_k\leq n$, then we shall sometimes write $z_I$ for $(z_{i_1},\ldots, z_{i_k})$.

\bigskip

We refer the reader to \cite{Sha} for more on composition operators and function spaces. Let us just mention that the problem we are dealing with has been solved for the Hardy space over the Euclidean unit ball $\mathbb B_n$ (\cite{Wog}) and later extended on the weighted Bergman spaces over $\mathbb B_n$ (\cite{KS}).

\section{Preliminary and main results}
\subsection{Upper and lower bounds for sublevel sets}
\begin{lemma}[Upper bound]\label{lem:est} Let $f:\mathbb D^n\to \mathbb D$ be holomorphic and $\mathcal C^1$-smooth on the closed polydisc. Suppose that $f(\zeta)=\eta$ for some $\zeta\in \TT^n$ and $\eta\in \TT$. Suppose also that $z_j\mapsto f(\zeta_1,\ldots, z_j, \ldots ,\zeta_n)$ is not constant for all $j=1,\ldots, n$. Then
\begin{equation}\label{eq:m}\vol(\{z\in \mathbb D^n\cap U:\ |f(z)-\eta|\leq \delta\})\preceq \delta^{n+1}
\end{equation}
for a sufficiently small neighbourhood $U$ of $\zeta$.
\end{lemma}

Note that the assumption of Lemma~\ref{lem:est} somehow means that $f$ does not depend on $n-1$-variables (otherwise the estimate would not be true).
We need the following result:

\begin{lemma}\label{lem:mob}
	Let $K$ be a compact set and $\varphi\in \mathcal C(\mathbb D\times K, \mathbb D)$ be such that $\varphi(\cdot, k)\in \mathcal O(\DD, \DD)$ for any $k\in K$. Then there exists a constant $C=C(K)>0$ such that for any $\delta\in (0,1)$ and $k\in K$ the following holds
	$$|x|\leq 1-\delta\ \Longrightarrow\ |\varphi (x, k)|\leq  1- C(K) \delta.$$
\end{lemma}

\begin{proof}
	The assertion is rather clear as an immediate consequence of the Schwarz lemma: put $\alpha_k=\varphi(0,k)$. Then for $|x|\leq 1- \delta$ elementary computations imply that $1-|\varphi(x)|^2\geq \delta |1 - \bar \alpha_k \varphi(x)|^2(1- |\alpha_k|^2)^{-1}\geq \delta (1-|\alpha_k|)/2$.
\end{proof}

\begin{proof}[Proof of Lemma~\ref{lem:est}] We can assume that $\zeta=(1,\ldots, 1)$ and $\eta=1$. Then, by the Julia-Carath\'eodory theorem, the assumption implies that $\frac{\partial f}{\partial z_j}(\zeta)> 0$, $j=1,\ldots, n$.
Clearly, $|f(z)-1|\leq \delta$ implies that $1-\delta \leq |f(z)|.$ It follows from the assumption that for $w\in \overline \DD^n$ close to $\zeta$ and fixed $j$ the function $z_j\mapsto f(w_1,\ldots, z_j, \ldots w_n)$ is in $\mathcal O(\DD,\DD)$ (as it is non-constant). Thus Lemma~\ref{lem:mob} applied to them guarantees that there is $D>0$ such that $|z_j|>1 - D\delta$ for $j=1,\ldots, n$, as $\delta \to 0$. Therefore, volumes of the projections of the set that appears in \eqref{eq:m} onto $z_j$-planes can be estimated from above by $1-(1-D\delta)^2 \simeq \delta$. We use this estimation for $j=1, \ldots, n-1$.

To get the assertion it is thus enough to get an estimate on $z_n$. Fix $z_1,\ldots, z_{n-1}$ and denote $\varphi(\lambda) = f(z_1, \ldots, z_{n-1},\lambda)$. Let $U_n$ be the projection of $U$ onto the $z_n$-plane. Take any $w_n\in \DD\cap U_n$ such that $ |\varphi(w_n)-1|< \delta $ (if there is no such $w_n$, the estimate we want to achieve is trivial). Clearly, it is enough to estimate the volume of a bigger set $\{\lambda\in U_n\cap \DD:\ |\varphi(\lambda) -\varphi(w_n)|< 2\delta\}$ which, in turn, can be estimated from above by the volume of the disc $\{\lambda:\ |\lambda - w_n|\leq \delta\}$, as the derivative $\varphi'(w_n)$ does not vanish (providing that $U$ is sufficiently small). This gives us an upper bound $\delta^2$ and the proof is finished.
\end{proof}

As we shall observe later (see Proposition~\ref{prop:beg}) the upper bound that we have just achieved is sharp.

\begin{lemma}[Lower bound]\label{lem:lb}
	Let $f:\DD^n\to \DD$ be holomorphic, $\mathcal C^2$-smooth on $\overline \DD^n$, and such that $f(\zeta)=\eta$ for some $\zeta\in \TT^n$ and $\eta\in \TT$.  Then 
	$$\vol(\{z\in \DD^n:\ |f(z) - \eta|\leq \delta\})\succeq \delta^{(3n+1)/2}.$$
\end{lemma}

\begin{proof} We lose no generality assuming that $\zeta=(1,\ldots, 1)$ and $\eta=1.$
The proof is a consequence of the following:
\begin{claim}\label{claim:up} Under the assumption of Lemma~\ref{lem:lb} there is a neighbourhood $U$ of $\zeta=(1,\ldots, 1)$ and a constant $C>0$ such that 
	\begin{equation*}\label{eq:f}|f(z)-1| \leq C\left| \frac{\partial f}{\partial z_1}(\zeta) (z_1 -1) + \ldots + \frac{\partial f}{\partial z_n}(\zeta) (z_n-1) \right|\text{ for $z\in \DD^n\cap U$}.
	\end{equation*}
\end{claim}
\begin{proof}[Proof of the Claim]
As we shall see the above inequality is a direct consequence of a multidimensional Julia-Carath\'eodory type lemma in the shape that was formulated by Bayart (Corollary~2.7, \cite{Bay}). We shall explain how to use it. One needs to show that there is a constant $C>0$ independent of $t>0$ such that if $\left| \frac{\partial f}{\partial z_1}(\zeta) (z_1 -1) + \ldots + \frac{\partial f}{\partial z_n}(\zeta) (z_n-1) \right|\leq t$, then $|f(z)-1|\leq C t$. 
	
It follows from the Julia-Carath\'eodory theorem that $\alpha_j:=\frac{\partial f}{\partial z_j}(\zeta)\geq 0$. Permuting the variables one can assume that $\alpha_j$ is strictly positive for $j=1, \ldots, k$, where $k\leq n$,  and $\alpha_j=0$ if $j>k$. The crucial part due to Bayart says that $$f(z) -1 = \sum_{j=1}^k \alpha_j (z_j-1) + \sum_{j=1}^k O(|z_j -1|^2).$$ Clearly, $|\alpha_1 (z_1-1) + \ldots + \alpha_k (z_k-1)|\leq t$ implies that $\Re (1-z_j)<D t$, where $D$ is independent of $t$. The last fact, together with $|z_j|<1$, implies that $|z_j-1|\leq D' \sqrt t$, where again, a constant $D'$ does not depend on $t$. This observation finishes the proof.
\end{proof}
To finish the proof of Lemma~\ref{lem:lb} it is enough to estimate the volume of the set $\{z\in \DD^n:\ |\alpha_1 (z_1-1) + \ldots + \alpha_k (z_k-1)|\leq \delta\}$. That was done by Bayart \cite{Bay} (see also Proposition~\ref{prop:beg} in the current paper) --- this volume behaves like $\delta^{(3k+1)/2}\succeq \delta^{(3n+1)/2}$.
\end{proof}

The estimates obtained above cannot be in general improved:

\begin{proposition}\label{prop:beg} Let $f,g:\mathbb D^n\to \mathbb D$ be given by the formulas $f(z)=z_1\cdots z_n$, $g(z_1, \ldots, z_n) = \alpha_1 z^n_1+\ldots + \alpha_n z^n_n$, where $\sum_j |\alpha_j|=1$, and $\alpha_j\neq 0$ for $j=1,\ldots, n$. Then for any $\eta \in \TT$: 
$$V_\beta (\{z\in \DD^n:\ |f(z)-\eta|\leq \delta\})\simeq \delta^{n(\beta + 1)+1},$$ while $$V_\beta(\{z \in \DD^n:\ |g(z)-\eta|\leq \delta\})\simeq \delta^{n(\beta +1) + (n+1)/2}.$$	
\end{proposition}

\begin{proof}
	The second estimate is essentially due to Bayart (see \cite{Bay}, Example~4.1), just a standard change of variables is additionally needed. Let us sketch the first one. We can assume that $\eta=1$. Clearly, $|1-z_1\cdots z_n|\leq \delta$ implies that $|z_1\cdots z_n|\geq 1- \delta$ which is somehow equivalent to $|z_j|\succeq 1-\delta$, $j=1,\ldots, n$. In this way we estimate projections on the $z_2$-$,\ldots, z_n$-planes and the estimation for $z_1$ follows from the fact that $\{z_1:\ |1-z_1\cdots z_n|\leq \delta\}$ is a disc.
\end{proof}

The bounds derived above show the following simple result that we will make use of in the sequel:

\begin{corollary}\label{cor:sharp}
Let $\Phi(\varphi):=(\varphi,\ldots, \varphi,0):\DD^n\to \DD^n$, where $\varphi$ is holomorphic on $\DD^n$. Let $f,g$ be as in Proposition~\ref{prop:beg}. Then for $\varphi=f$ the composition operator $C_{\Phi(f)}$ is bounded on $A_\beta^2(\DD^n)$ precisely when $n\leq \beta +3$.

On the other hand, for $\varphi=g$ the operator $C_{\Phi(g)}$ is bounded on $A_\beta^2(\DD^n)$ precisely when $n/2\leq \beta +5/2$.
\end{corollary}

\begin{proof}
Estimating $V_\beta(\Phi(f)^{-1}(S(\zeta, \delta))$ we can assume that $\delta_n=1$, that is $\delta=(\delta_1,\ldots, \delta_{n-1}, 1)$. What we need to do is to check when the inequality $V_\beta(\{z:\ |f(z)-\zeta|<\delta_i,\ i=1, \ldots, n-1\}) \preceq \delta_1^{2+\beta} \cdots \delta_{n-1}^{2+ \beta}$ is fulfilled for any $\delta_1, \ldots , \delta_{n-1}>0$, which is equivalent to checking if $V_\beta(\{z:\ |f(z)-\zeta|<\delta_1\}) \preceq \delta_1^{(2+\beta)(n-1)}$ for $\delta_1>0$. This, due to Proposition~\ref{prop:beg}, boils down to a trivial inequality $\delta_1^{n(\beta+1)+1} \preceq \delta_1^{(2+\beta)(n-1)}$. 

The proof for $g$ is exactly the same.
\end{proof}

\subsection{Rank-sufficiency theorem and linear-necessity conditions}

\begin{theorem}\label{thm:3.2}
	Suppose that $\Phi:\mathbb D^n\to \DD^n$ is holomorphic and $\mathcal C^1$ smooth on $\overline{\mathbb D}^n$ and such that for any $q$ and any $|I|=q$:
	
	$$(\dag)\quad  \forall \zeta\in \TT^n \text{ such that }\Phi_I(\zeta)\in \TT^q:\quad  \rank d_\zeta \Phi_I=q.$$
	Then for any $\beta>-1$ the composition operator $C_\Phi$ maps $A_\beta^2$ as well as $H^2$ continuously into itself.
\end{theorem}
We shall postpone the proof of Theorem~\ref{thm:3.2} to the end of the paper. Note that it provides us with a simple criterion for the weighted Bergman and Hardy spaces over the bidisc.

\begin{theorem}\label{th:bidisc}
Let $\phi\in \mathcal O(\DD^2, \DD^2)\cap \mathcal C^2 (\overline\DD^2)$, $\beta>-1$. Then $C_\phi$ is bounded on $A_\beta(\DD^2)$ (resp. on $H^2(\DD^2)$) if and only if $d_\zeta \phi$ is invertible for all $\zeta\in \TT^2$ such that $\phi(\zeta)\in \TT^2$.
\end{theorem}

\begin{proof} Necessity follows from \cite{KSZ}, while sufficiency is a direct consequence of Theorem~\ref{thm:3.2}.
\end{proof}

\begin{corollary}\label{thm:FO}
First order conditions are sufficient for $H^2(\DD^n)$ and for $A_\beta^2(\DD^n)$ with $\beta<0$ if and only if $n= 2$.

\end{corollary}

\begin{proof} In view of Theorem~\ref{th:bidisc} it is enough to show that the first order conditions fail for $n\geq 3$. Let $f$ be as in Proposition~\ref{prop:beg}. If $n>2$, then $C_{\Phi(f)}$ is not bounded on $A_\beta^2(\DD^n)$ or $H^2(\DD^n)$. Nevertheless, one can check that it does satisfy the first order conditions. The last fact can also be deduced in the following way: for any $\zeta\in \TT^n$ such that $f(\zeta)\in \TT$ there is $g$ as in Proposition~\ref{prop:beg} such that $g(\zeta)\in \TT$ and the derivatives $d\Phi(f)$ and $d \Phi(g)$ coincide at $\zeta$.
\end{proof}

Therefore the bidisc and the tridisc are very particular for the problem of boundedness of composition operators on the Bergman space. According to Theorem~\ref{th:bidisc} the sufficient and necessary condition is quite simple to formulate in the case of the bidisc. As we shall see below, it turns out it is also the case for the tridisc. 

\begin{theorem}\label{th:tridisc}
Let $\phi:\DD^3\to \DD^3$ be holomorphic on $\DD^3$ and $\mathcal C^2$-smooth on $\overline{\DD}^3$. Then $C_\phi$ is bounded on $A^2(\DD^3)$ if and only if $\phi$ satisfies the following properties:
\begin{enumerate}
	\item for any $\zeta\in \TT^3$ such that $\phi(\zeta)\in \TT^3$ the derivative $d_\zeta \phi$ is invertible; and
	\item for $1\leq i_1<i_2\leq 3$ and any $\zeta\in \TT^3$ such that $(\phi_{i_1}(\zeta), \phi_{i_2}(\zeta)) \in \TT^2$ either 
	
	a) $\nabla \phi_{i_1}(\zeta)$, $\nabla \phi_{i_2} (\zeta)$ are linearly independent, or 
	
	b) $\frac{\partial \phi_{i_k}}{\partial z_j}(\zeta)\neq 0$ for $j=1,2,3$ and $k=1,2$. 
\end{enumerate}
\end{theorem}

The theorem above says that in the settings of $A^2(\DD^3)$ it is possible to characterize boundedness of the operator $C_\phi$ in terms of the derivative of $\phi$ on $\TT^3$. We already knew that was the case for $A^2(\DD^2)$. The following corollary shows that there is no hope to achieve such characterization for $A^2(\DD^n)$ when $n>3$.

\begin{corollary}
First order conditions are sufficient for $A^2(\DD^n)$ if and only if $n\leq 3$.
\end{corollary}

\begin{proof}
"If" follows from Theorem~\ref{th:tridisc}. To prove "only if" we proceed as in the proof of Theorem~\ref{thm:FO}. Keeping the notation from there we see that for $n=4$, the operator $C_{\Phi(f)}$ is not bounded on $A^2(\DD^n)$, while it satisfies the first order conditions.
\end{proof}

\begin{proof}[Proof of Theorem~\ref{th:tridisc}]

Let us prove the necessity. Seeking a contradiction assume that {\it 1.} and {\it 2.} are not satisfied. Take $\zeta\in \TT^3$. If $\phi(\zeta)\in \TT^3$, we use the necessity from \cite{KSZ}. So suppose that, up to a permutation of components, $\phi(\zeta)\in \TT^2 \times \DD$, and condition {\it 2.} is not satisfied, i.e. $\nabla \phi_1(\zeta), \nabla \phi_2 (\zeta)$ are linearly dependent and $\frac{\partial \phi_1}{\partial z_j}(\xi)$ vanishes for some $j=1,2,3$. Then the assertion follows from the Claim~\ref{claim:up}. To see this we can make some reductions: we can assume that $\zeta=(1,1,1)$ and $\phi_1(\zeta)= \phi_2(\zeta) =1$. Let $\delta = (\delta_1, \delta_1, 1)$. Denote $\alpha_j = \frac{\partial \phi_1}{\partial z_j}(\zeta)\geq 0$. Then it follows from Claim~\ref{claim:up}, and the fact that the derivatives of $\phi_1$, $\phi_2$ are linearly dependent, that there is $C>0$ such that $\{z\in \DD^3:\ |\sum \alpha_j (z_j -1) | \leq \delta_1\}$ is contained in $\{z\in \DD^3:\ |\phi_i(z)-1|< C \delta_1,\ i=1,2\}$. As at least one of $\alpha_j$ vanishes, we infer that $\vol(\Phi^{-1}(S(\zeta, \delta)))\succeq \delta_1^{7/2}$. Since $\vol(S(\zeta, \delta))\simeq \delta_1^4$, the necessity follows.

Before we start the proof of the sufficiency let us observe that the assumptions of the theorem, that are two conditions {\it 1.} and {\it 2.}, are formulated only for $\zeta\in \TT^3$. However, Lemma~\ref{lem:obs} assures us that they remain true if $\zeta\in \overline\DD^3$. Note also that, by the Julia-Carath\'eodory theorem, $\phi$ automatically satisfies one additional property: 

{\it 3. if, up to a permutation of the components, $\phi(\zeta)\in \TT\times \DD^2$, the rank of the derivative $d_\zeta \phi_1$ is $1$.}

It is a trivial observation that sets composed of points $\zeta\in \partial {\DD}^n$ that satisfy any of conditions {\it 1, 2a), 2b), 3} are open and contain $\phi^{-1}(\partial \DD^n)$. We shall cover them with a finite number of domains that posses particular properties. 

Let $\{U^1_\iota\}$ be finitely many polydiscs that cover the set $\phi^{-1}(\TT^3)$ such that the derivative $d_\zeta \phi$ is invertible on $\overline{U}_\iota^1\cap \overline \DD^3$. Let $\epsilon>0$ be such that for any $z\in \partial \DD^3\setminus \bigcup_i U_i^1$ there exists $i=1,2,3,$ such that $|\phi_i(z)|<1-\epsilon$.

We may take a finite number of polydiscs $\{U_\iota^2\}$ that cover the set $\phi^{-1}((\TT^2\times \DD) \cup (\TT\times \DD \times \TT) \cup (\DD \times \TT^2))\setminus \bigcup U^1_\iota$ and such that one of conditions {\it 2a.} or {\it 2b.} is satisfied on every $\overline U_\iota^2 \cap \DD^3$. We can additionally assume that every $U_\iota^2$ satisfies the following: there is a permutation $i_1, i_2, i_3$ of $1,2,3$ such that  $|\phi_{i_3}(z)|<1-\epsilon$ on $\overline U_\iota^2\cap \DD^3$ and 

-- either there are $j_1, j_2$ such that $\det \frac{\partial(\phi_{i_1}, \phi_{i_2})}{\partial (z_{j_1}, z_{j_2})}\neq 0$, or

-- $\frac{\partial \phi_{i_k}}{\partial z_j}\neq 0$, $j=1,2,3$ and $k=1,2.$

Finally we define a finite family of polydiscs $\{U_\iota^3\}$ that cover $\phi^{-1} ((\TT\times \DD^2)\cup (\DD \times \TT \times \DD) \cup (\DD^2\times \TT))\setminus (\bigcup U_\iota^1\cup \bigcup U_\iota^2)$ such that on every $\overline U_\iota^3\cap \overline \DD^3$ there is a permutation $i_1, i_2, i_3$  and $j=1,2,3$ such that  $\frac{\partial \phi_{i_3}}{\partial z_j}\neq 0$ and $|\phi_{i_2}(z)|, |\phi_{i_3}(z)|<1-\epsilon$ for $z\in U_\iota^3\cap \overline \DD^3$ (we decrease $\epsilon$, if necessary). Decreasing $\epsilon$ again we can assume that 
\begin{equation}\label{eq:eps}|\phi_j(z)|<1-\epsilon
\end{equation}
for $j=1,2,3$ and $z$ lying out of the constructed covering.

We aim at proving that $\vol (\phi^{-1}(S(\zeta, \delta))) \preceq \vol(S(\zeta, \delta))$ for $\zeta \in \TT^3$ and $\delta\in (0,1)^3$. By \eqref{eq:eps} it is enough to show that
\begin{equation}\label{eq:123} \vol (\phi^{-1}(S(\zeta, \delta))\cap U_\iota^i)\leq C_{i, \iota} \vol(S(\zeta, \delta)) \text{ for any } i,\iota.
\end{equation}

For $i=1$ inequality \eqref{eq:123} follows directly from Corollary~\ref{cor:de}. Suppose that $i=2$ and the condition {\it 2.a)} is satisfied on $U_\iota^2$. Permuting the variables and components we can assume that $|\phi_3|<1-\epsilon$ and $\det \frac{\partial (\phi_1, \phi_2)}{\partial (z_1, z_2)} \neq 0$ on $\overline U_\iota^2\cap \overline \DD^3$. Write $U_\iota^2=U'\times U''\subset \CC^2\times \CC$. Fixing $z_3\in U''$ and applying Corollary~\ref{cor:de} to $\psi(z_1, z_2) = \phi(z_1, z_2, z_3)$ we get the assertion in this case. We proceed in the same way if $i=3$.

We are left with proving inequality \eqref{eq:123} in the case when condition {\it 2.b} holds on $U_\iota^2$. Here we can apply Lemma~\ref{lem:est} to $\phi_{i_1}$ and $\phi_{i_2}$ obtaining that $\vol (\phi^{-1}(S(\zeta, \delta)) \cap U_\iota^2)\preceq \min (\delta_{i_1}^4, \delta_{i_2}^4)\leq \delta_{i_1}^2 \delta_{i_2}^2.$ This finishes the proof.

\end{proof}

\begin{lemma}\label{lem:3.5}
Suppose that $\Phi\in \mathcal O(\DD^n, \DD^n)\cap \mathcal C^1(\overline{\mathbb D}^n)$ is such that for any $q$ and any $|I|=q$ the following holds:
$$(\dag)\quad  \forall \zeta\in \TT^n \text{ such that }\Phi_I(\zeta)\in \TT^q:\quad  \rank d_\zeta \Phi_I=q.$$
Then for any $\beta>-1$ there is a constant $C_\beta$ such that the composition operator $C_\Phi$ maps $A_\beta^2$ continuously into itself and its norm is bounded by $C_\beta$. Moreover, $C_\beta$ can be chosen independently of $\beta$ if $\beta<0$.
\end{lemma}

Note that (\dag) is stated just for points lying in the distinguished boundary of the polydisc. However, we shall show that the fact that it holds on the topological boundary is its consequence. In other words, (\dag) implies that for any $\xi\in \overline{\mathbb D}^n$ such that $\Phi_I(\xi)\in \TT^q$, the derivative $d_\xi \Phi_I$ has rank $q$. To prove this fact we shall need the following simple observation:

\begin{lemma}\label{lem:obs} Suppose that $\varphi \in \mathcal O(\DD^2, \DD)\cap \mathcal C^1(\overline {\DD}^2)$ and points $z_0\in \DD$, $\zeta\in \TT$ are such that $|\varphi(z_0,\zeta)|=1$. Then $\frac{\partial \varphi}{\partial z}(z,\zeta)=0,$ $z\in \DD$, and $z\mapsto \frac{\partial \varphi}{\partial w}(z,\zeta)$ is constant.

In particular, if $\psi\in\mathcal O(\DD^n, \DD)\cap \mathcal C^1(\overline \DD^n)$ and $(z_0,\zeta)\in \DD^m\times \TT^{n-m}$ are such that $|\psi(z_0,\zeta)|=1$, then $z\mapsto \nabla\psi (z, \zeta)$ is constant on $\DD^m$.
\end{lemma}

\begin{proof} We can assume that $\zeta=1$ and $\varphi(z_0, 1)=1$. Note that the first statement is trivial, as $z\mapsto \varphi(z,1)$ is constant. The second one is, in turn, a consequence of this simple observation and the Julia-Carath\'eodory lemma applied to $w\mapsto \varphi(z,w)$, which tells us that for any $z$ the derivative $\frac{\partial \varphi}{\partial w}(z,1)$ is positive. Since $z \mapsto \frac{\partial \varphi}{\partial w}(z,1)$ is analytic, it must be constant.
\end{proof}

\begin{proof}[Proof that $(\dag)$ holds on the topological boundary]
Take $q$, $I$ and $\xi=(z',\zeta'')\in \DD^m\times \TT^{n-m}$ such that $\Phi_I(\xi)\in \TT^q$.
The assertion is a consequence of Lemma~\ref{lem:obs}, according to which $ d_{(z', \zeta'')} \Phi_I =  d_{(\zeta', \zeta'')} \Phi_I$ for any $z'\in \DD^m$ and $\zeta'\in \TT^m$.
\end{proof}

\begin{proof}[Proof of Lemma~\ref{lem:3.5} for $\beta  \in (-1,0)$] 
Decompose the topological boundary of the polydisc onto $n$-sets that depend on the complex dimension: $b_0\DD^n:=\TT^n$, $b_1\DD^n :=(\TT^{n-1}\times \DD) \cup \ldots (\DD\times \TT^{n-1}),\ldots, b_{n-1}\DD^n:=(\TT\times \DD^{n-1}) \cup \ldots \cup (\DD^{n-1}\times \TT).$

Let $\mathcal U_0$ be a finite covering of $\Phi^{-1}(b_0\DD^n)$ by domains such that the derivative $d_z \Phi$ is invertible on $\overline U\cap \overline \DD^n$ on every $U \in \mathcal U_0$. Take $\epsilon>0$ such that for any $z\in \partial \DD^n\setminus \bigcup \mathcal U_0$ there is $i=1,\ldots, n,$ such that $|\Phi_i(z)|<1-\epsilon$. 

Let us explain how we build a finite covering $\mathcal U_1$ of $\Phi^{-1}(b_1 \DD^n)\setminus\bigcup \mathcal U_0$. It will be composed of domains $U_1$ that are to be constructed below. To make this covering finite it suffices to use a compactness argument. Take $z\in \Phi^{-1}(b_1\DD^n)\setminus \bigcup \mathcal U_0$. Permuting the variables and components we can assume that $\Phi(z)\in \TT^{n-1}\times \DD$ and $\det\frac{\partial(\Phi_1, \ldots, \Phi_{n-1})}{\partial (z_1, \ldots, z_{n-1})}(z)\neq 0$. We define $U_1=U_1'\times U_1''\subset \CC^{n-1} \times \CC$ to be a neighbourhood  of $z$ such that the last inequality remains valid on $\overline U_1\cap \overline \DD^n$. Shrinking it we can also assume that $|\Phi_n|<1-\epsilon$ on $U_1$. 

We define finite families $\mathcal U_2, \ldots, \mathcal U_{n-1}$ composed of domains covering sets $\Phi^{-1}(b_2 \DD^n),\ldots, \Phi^{-1}(b_{n-1} \DD^n)$ in a similar way (decreasing $\epsilon$, if necessary). 

We shall show that for any $\delta\in(0,1)^n$ the inequality 
\begin{equation}\label{eq:vb} V_\beta (\Phi^{-1}(S(\zeta, \delta))\cap U) \preceq V_\beta (S (\zeta, \delta))
\end{equation} holds whenever $U\in \mathcal U_{n-k}$ and $k=1,\ldots, n$. Fix $k$. Permuting the variables and components we can assume that 
$\det \frac{\partial (\Phi_1,\ldots, \Phi_k)}{\partial(z_1,\ldots, z_k)} \neq 0$ on $\overline U \cap \overline \DD^n$ and that $|\Phi_{k+1}|,\ldots, |\Phi_n|<1-\epsilon$ there. Therefore, we need to prove \eqref{eq:vb} for $\delta_{k+1}, \ldots, \delta_n$ greater then $\epsilon$. In other words, it is enough to prove that $ V_\beta (\Phi^{-1}(S(\zeta, \delta))\cap U)\preceq \delta_1^{2+\beta} \cdots \delta_k^{2+\beta}$. Write $U=U'\times U''\subset \CC^k \times \CC^{n-k}$.

By the Fubini theorem it is sufficient to get the following estimate
\begin{multline}\label{eq:vbeta} V_\beta (\{(z_1,\ldots, z_k)\in U'\cap \DD^k:\ |\Phi_i(z', z'') - \zeta_i|<\delta_i,\ i=1,\ldots, k\})\leq \\ C \delta_1^{2+\beta} \cdots \delta_k^{2+\beta},
\end{multline}
where $z''\in U''$ and $C$ is a constant that does not depend on $\delta_j$ and $z''\in \overline U''$.
To proceed we need a few preparatory lemmas.

\begin{lemma}\label{lem:de}
	Let $\psi\in \mathcal O(\DD^k,\CC^k) \cap \mathcal C^1(\overline \DD^k)$,  and let $V$ be a domain such that the derivative $d_z\psi$ is non-degenerate for $z\in \overline V\cap \overline\DD^k$. 
	Let $C>0$ be a constant that bounds the norm of $\psi$ on $\overline V \cap \DD^k$, the norm of the derivative of $\psi$ on $\overline V\cap \overline \DD^k$, and the norm of derivative of $\psi^{-1}$ on $\psi(\overline V\cap \overline \DD^k)$.
	
	Then there are constants $c_1,\ldots, c_k>0$ depending on $C$ such that for every $z\in \overline V\cap \overline \DD^k$ there are a neighbourhood $U$ of $z$, a permutation $\sigma$ of the set of $k$-elements, and there are  $z_{\sigma(1)}'$, $z_{\sigma(2)}'(z_{\sigma(1)}),$ $\ldots,$ $z'_{\sigma(k)}(z_{\sigma(1)}, \ldots, z_{\sigma(k-1)})$ such that for any $\delta_1,\ldots, \delta_k>0$ the set $\{z\in U\cap \DD^k:\ |\psi_i(z)|<\delta_i:\ i=1,\ldots, k\}$ is contained in $$\{z\in \DD^k:\ |z_{\sigma(1)}-z'_{\sigma(1)}|<c_1\delta_{i_1},\ldots,  |z_{\sigma(k)} - z'_{\sigma(k)}(z_{\sigma(1)}, \ldots, z_{\sigma(k-1)})|<c_k \delta_{i_k}\},$$
	where $i_1,\ldots, i_k$ is a rearrangement of $1, \ldots, k,$ such that $\delta_{i_1}\geq \ldots \geq \delta_{i_k}$. 
\end{lemma}

Let us postpone for a moment the proof of the lemma. Its consequence is the following:

\begin{corollary}\label{cor:de}
 We keep the notation from Lemma~\ref{lem:de}. If $\beta\in (-1,0]$, then there is a constant $c$ depending only on $C$ and a number of sets $U$ covering $\overline V$ such that $$V_\beta(\{z\in \overline V\cap\overline \DD^k:\ |\psi_i(z)|<\delta_i:\ i=1,\ldots, k\})\leq c \delta_1^{2+ \beta} \cdots \delta_k^{2+\beta}.$$
\end{corollary}

The above corollary can be deduced using a standard compactness argument, the Fubini theorem, and the following simple result:

\begin{lemma}\label{lem:ad}
	Let $\beta\in (-1,0]$ and $a\in\CC$. Then there is a constant $C$ independent of $\delta$, $\beta$, and $a$ such that $A_\beta(\DD(a, \delta)\cap \DD) \leq C \delta^{\beta + 2}.$
\end{lemma}

\begin{proof}[Proof of Lemma~\ref{lem:ad}] If $|a|<1/2$, then $A_\beta(\DD(a, \delta))\simeq \delta^2$ and the result is clear. If, in turn,  $|a|\geq1$, observe that we can find $a'\in \TT$ such that $\DD(a, \delta)\cap \DD$ is contained in $\DD(a',\delta)$. Then, it is enough to check that $A_\beta(\DD(a',\delta)\cap \DD)\simeq \delta^{\beta +2}$.
	
Suppose that $1> |a|>1/2$. Rotating we can assume that $1> a>1/2$. Then $\DD(a, \delta) \subset \DD(a+\delta, 2\delta) \cap \DD (0, a+\delta)$. Due to the case considered above we are able to estimate $A_\beta(\DD(a+\delta, 2\delta)\cap \DD)$ if $a+\delta\geq 1$. Thus, it is enough to focus on the remaining case $a+\delta<1$. Since $\beta<0$, the inequality $(1-|z|^2)^\beta \leq ((a+\delta)^2 - |z|^2)^\beta$ holds for $|z|<a+\delta$. Consequently, 
	\begin{multline*} A_\beta(\DD(a, \delta)) = (\beta+1)\int_{\DD(a, \delta)} (1-|z|^2)^\beta dA(z) \leq\\ (\beta+1)\int_{\DD(a+\delta, 2\delta) \cap \DD (0, a+\delta)} ((a+\delta)^2-|z|^2)^\beta dA(z).
	\end{multline*}
	A change of variables argument (precisely, $z\mapsto z/(a+\delta)$) applied to the last integral shows that $A_\beta(\DD(a, \delta)) \preceq A_\beta(\DD(1, 4\delta)\cap \DD) \simeq \delta^{2+\beta}$, as claimed.
\end{proof}

\begin{proof}[Proof of Lemma~\ref{lem:de}]
We shall construct constants $c_j$ and it will be clear that they depend only on $C$. The argument is standard and follows by an induction. Let us explain how the inductive step works.

Denote $S:=\{z\in \overline\DD^k:\ |\psi_i(z)|\leq \delta_i:\ i=1,\ldots, k\}$. Permuting the components of $\psi$ we can assume that $\delta_1 \geq \ldots \geq \delta_k$.
Fix a point in $S\cap \overline V$ (if $z\notin S$, the existence of a claimed neighbourhood $U$ is clear). Since at this point the derivative $d\psi$ is non-degenerate, expanding its Jacobian determinant by minors we see that there is a permutation of $k-1$-element set such that
	\begin{equation}\label{eq:ja}
\det \frac{\partial (\psi_2,\ldots, \psi_{k})}{\partial (z_{\sigma(2)}, \ldots, z_{\sigma(k)})}\neq 0,
\end{equation} on its neighbourhood.
In this way we choose $\sigma(1)$. Other elements $\sigma(j)$, $j\geq 2$, are to be defined in further inductive steps after a possible remuneration. Take a neighbourhood $U$ of the fixed point in $S\cap \overline V$ such that inequality \eqref{eq:ja} holds on $\overline U\cap \overline\DD^k$. 
	
Let $z,w\in U \cap S$.  Since $\psi$ is diffeomorphic (shrink $U$, if necessary), $||\psi(z)-\psi(w)|| \geq C^{-1}||z-w||\geq C^{-1} |z_{\sigma(1)} -w_{\sigma(1)}|$. In particular, $$|z_{\sigma(1)} - w_{\sigma(1)}| \leq c_1 \delta_1,$$ where $c_1$ depends just on $C$ (one can take $c_1= 2 \sqrt{k} C$ here).
	
Write $z=(z_{\sigma(1)},z')\in \CC \times \CC^{k-1}$. Since \eqref{eq:ja} holds on $U$, shrinking it we can assume it is of the form $U=U_1\times U'\subset \CC\times \CC^{k-1}$ and that for any $z_{\sigma(1)}\in \overline U_1\cap \overline \DD$ the map $$U'\cap \DD^{k-1}\ni z' \mapsto (\psi_2(z_{\sigma(1)},z'),\ldots, \psi_{k}(z_{\sigma(1)},z'))$$ is a diffeomorphism and the bound on the derivative and its inverse is uniform with respect to $z_{\sigma(1)} \in \overline U_1\cap \overline \DD$ and depends only on a constant $C$. These maps and the set $S':=\{z'\in \overline \DD^{k-1}:\ |\psi_j( z_{\sigma(1)},z')|<\delta_j,\ j=2,\ldots, k\}$ are the objects we are applying the inductive process to. To be more precise, after a possible permutation of elements $\sigma(2), \ldots, \sigma(k)$ we get that $\det \frac{\partial (\psi_3,\ldots, \psi_{k})}{\partial (z_{\sigma(3)}, \ldots, z_{\sigma(k)})}(z)\neq 0$ for $z\in \overline U\cap \overline\DD^k$ (shrink $U$, if necessary). By choosing this permutation we define $\sigma(2)$. Then, we show exactly as above that for points $(z_{\sigma(1)},z'), (z_{\sigma(1)},w') \in U$, where $z'=(z_{\sigma(2)}, \ldots, z_{\sigma(k)})$ and $w'=(w_{\sigma(2)}, \ldots, w_{\sigma(k)})$, the estimate $$|z_{\sigma(2)} - w_{\sigma(2)}|\leq c_2 \delta_2 $$ holds with a constant $c_2$ depending only on $C$.
	
	Sets $U$ we are looking for are obtained after a possible shrinking at every step of carrying out of this process. 
	
	The assertion is a consequence of what we have just proven.
	
\end{proof}

We are ready to finish the proof of Lemma~\ref{lem:3.5}. We have boiled it down to inequality \eqref{eq:vbeta} which turns out ot be a consequence of the Fubini theorem and Corollary~\ref{cor:de} applied to $\psi:=\Phi_I(\cdot, z'')-\zeta_I$, where $\zeta_I=(\zeta_1,\ldots, \zeta_k)$ and $\Phi_I=(\Phi_1,\ldots, \Phi_k)$.

\end{proof}

\begin{proof}[Proof of Lemma~\ref{lem:3.5} for $\beta\geq 0$]

We shall start this part with the following observation: 
\begin{claim}\label{claim:SL} Let $\varphi\in \mathcal O(\DD^k, \DD^k)$ and $\zeta\in \TT^k$. Let $C$ be a positive constant and $U$ be a neighbourhood of $\zeta$ such that $|\det d_z\Phi|>C$ for $z\in U\cap \DD^k$. Then for $z\in U\cap \DD^k$ the following inequality holds:
	$$(1-|\Phi_1(z)|^2) \cdots (1- |\Phi_k(z)|^2) \geq \left(\frac{C}{k!}\right)^k (1-|z_1|^2) \cdots (1- |z_k|^2).$$
\end{claim}
\begin{proof}[Proof of the claim] To prove the claim recall that the Schwarz lemma says that $|\varphi'(\lambda)| (1-|\lambda|^2) \leq (1-|\varphi(\lambda)|^2)$ for $\lambda \in \DD$, whenever $\varphi$ is a holomorphic self-function of the unit disc. Fix $z'\in U\cap \DD^k$. The condition on the derivative implies that there is a permutation of a set of $k$-elements such that $\left|\prod_{j=1}^k \frac{\partial \Phi_{\sigma(j)}}{\partial z_j}(z) \right| \geq C/k!$. Applying the above-mentioned Schwarz lemma to analytic discs $z_j\mapsto \Phi_{\sigma(j)} (z'_1, \ldots, z_j,\ldots, z_k')$, $j=1,\ldots, k$, we get $k$ inequalities. Multiplying them out we get the claim.
\end{proof}

With this tools in hands to prove the main assertion of the lemma it is enough to follow almost line by line the proof of Theorem~3.1 of \cite{Bay}. Keeping all the notation from there let us just additionally define $\tilde m:=\min_I \tilde m_I$, where $\tilde m_I:=\min_{z\in \overline{\mathcal G_I}} \det(M_\Phi(I, J, z))$.

To get the assertion it is enough to show that there is $C>0$ such that 
$$V_\beta (\Phi^{-1}(S(\zeta, \delta))) \leq C V_\beta(S(\zeta, \delta)).$$ As in the proof of Theorem~3.1 this boils down to estimating the volume $$V_\beta (\{z'_J\in V_l'(\xi_l):\ |\Phi_i(z_J', z_J'') - \zeta_i|<\delta_i,\ i\in I\}).$$ Changing the variables and making use of Claim~\ref{claim:SL} we can estimate the above volume exactly as in the proof of Theorem 3.1 in \cite{Bay}.
\end{proof}

\begin{proof}[Proof of Theorem~\ref{thm:3.2}]
	For the weighted Bergman spaces the assertion follows from Lemma~\ref{lem:3.5}. To get it for the Hardy space it suffices to recall that if a constant $C_\beta$ in \eqref{eq:box} is independent of $\beta$, then \eqref{eq:box} implies the boundedness of the composition operator on the Hardy space.
\end{proof}

{\bf Acknowledgements.} I would like to thank the anonymous referee for numerous remarks that substantially improved the shape of the paper.

\end{document}